\documentclass[a4paper,10pt,oneside]{article}

\usepackage[english]{babel}
\usepackage[T1]{fontenc}
\usepackage{enumerate}
\usepackage{amsmath}
\usepackage{amsthm}
\usepackage{amssymb}
\usepackage{hyperref}

\newcommand{\Qp}{\mathbb{Q}_p}

\newcommand{\Zp}{\mathbb{Z}_p}
\newcommand{\Zpex}{\mathbb{Z}_{p,exp}}

\newcommand{\Cp}{\mathbb{C}_p}

\newcommand{\Q}{\mathbb{Q}}
\newcommand{\R}{\mathbb{R}}

\newcommand{\C}{\mathbb{C}}
\newcommand{\Z}{\mathbb{Z}}
\newcommand{\N}{\mathbb{N}}

\newcommand{\Lexp}{\mathcal{L}_{exp}}

\newcommand{\deriv}[2]{\frac{\partial{#1}}{\partial{#2}}}

\newtheorem{theorem}{Theorem}[section]

\newtheorem{prop}[theorem]{Proposition}
\newtheorem{algo}{Algorithm}
\newtheorem{lemma}[theorem]{Lemma}
\newtheorem{cor}[theorem]{Corollary}
\newtheorem{definition}[theorem]{Definition}

\newtheorem{claim}{Claim}
\newtheorem*{conj}{Conjecture}

\newtheorem*{algo*}{Algorithm}
\theoremstyle{remark}
\newenvironment{Remark}{\begin{trivlist}\item[\hskip \labelsep {\bfseries Remark.}]}{\end{trivlist}}

\newenvironment{notation}{\begin{trivlist}\item[\hskip \labelsep {\bfseries Notations.}]}{\end{trivlist}}
\newenvironment{MainTheorem}{\begin{trivlist}\item[\hskip \labelsep {\bfseries Main theorem.}]}{\end{trivlist}}

\title{$p$-adic exponential ring, $p$-adic Schanuel's conjecture and decidability}
\author{Nathana\"el Mariaule} 
\date{}
\begin{document}
\maketitle

\begin{abstract} Let $exp(x)$ be the power series $\sum x^n/n!$. Then $E_p(x):=exp(px)$ (if $p\not=2$, we set $E_2(x)=exp(4x)$) is well-defined on $\Zp$  and determines a structure of exponential ring. In this paper, we prove that if a $p$-adic version of Schanuel's conjecture is true then the theory of $(\Zp, +, \cdot, 0, 1, E_p)$ is decidable.

\end{abstract}

\section{Introduction}
Exponential rings and fields are well-established topics in model-theory. While the real exponential field $\R_{exp}$ and the complex exponential field $\C_{exp}$ have been studied by many people, there is pretty much nothing known about the $p$-adic numbers in the context of exponential fields.  Note that there is no clear structure of exponential field attached to $\Qp$. It is not hard to show that the natural exponential map $exp(x)$ determined by the power series $\sum x^n/n!$ is convergent iff $v_p(x)>1/(p-1)$. As there is no analytic continuation for analytic functions in $\Qp$, we cannot extend $exp(x)$ to a natural exponential function on $\Qp$. However, we can use $exp(x)$ to define a structure of exponential ring on $\Zp$: let $E_p$ be the map $x\longmapsto exp(px)$ (if $p\not=2$, we set $E_2(x)=exp(4x)$). Then $E_p$ is convergent on $\Zp$ and $(\Zp,+,\cdot , 0, 1, E_p)$ is an exponential ring i.e. $E_p$ is a morphism of group from $(\Zp,+,0)$ to $(\Zp^\times, \cdot, 1)$. Let us remark that the choice of $E_p$ is not really canonical: we could have taken $exp(tx)$ for any $t$ with positive valuation.

\par Let $\mathcal{L}_{exp}$ be the language $(+,\cdot, 0, 1, E_p)$ and $\Zpex$ be the $\mathcal{L}_{exp}$-structure with underlying set $\Zp$ and natural interpretation for the elements of the language. We call this stucture the $p$-adic exponential ring. $\Zpex$ can be thought as a $p$-adic equivalent to the structures $\C_{exp}$ and $\R_{exp}$ (or to be more precise we should consider these structure with the domain of the exponential map reduced to a compact set). Very few results are known about $\Zpex$. In \cite{Macintyre2}\cite{Macintyre3}, A. Macintyre investigates for the first time the stucture of $p$-adic exponential ring. He solves some decisions problem (decision problem for the positive universal part of the theory and for one variable existential sentences). In \cite{Macintyre4}, he also gives a result of model-completeness whose an effective version was proved in \cite{Mariaule1} and will be usefull for this paper.

\par In this paper, we consider the question of the full decidability of the theory of $\Zpex$. It is well-known that the theory of $\C_{exp}$ is undecidable. Indeed, using the kernel of the exponential map, one can define the structure $(\Z, + ,\cdot , 0, 1)$ inside $\C_{exp}$. On the other side, the decidability problem for the theory of $\R_{exp}$ remains open. In \cite{Macintyre-Wilkie}, A. Macintyre and A. Wilkie highlight the role of Schanuel's conjecture in this problem. More precisely, they prove that if Schanuel's conjecture is true then the theory of $\R_{exp}$ is decidable. The main result of this paper is along this line:
\begin{MainTheorem}Assume that the $p$-adic version of Schanuel's conjecture is true. Then the theory of the $p$-adic exponential ring is decidable.
\end{MainTheorem}

\par The outline of the proof is the same that in the real case: first we reduce the decidability of the full theory to the case of existential sentences. In the real case, it relies (among other things) on the effective model-completeness of the theory of the reals in the language of ordered fields with the restricted exponential function. In the $p$-adic case, it is not known if the theory of $\Zpex$ is model-complete. So, we will use the expansion of language $\mathcal{L}_{pEC}$ introduced in \cite{Macintyre4}. In this language the theory of $\Zp$  is effectively model-complete. We recall the definition of this language in the beginning of section \ref{Decidability of the LPEC-sentences}. While this language introduces new functions, it is not so different from the language $\mathcal{L}_{exp}$ from the point of view of the decision problem for existential sentences. Indeed, the new functions can be written as polynomial combination of exponentials in a suitable finite algebraic extension of $\Qp$. Therefore, in section \ref{Decidability of positive existential sentences}, we will show that one can solve the decision problem for the positive existential $\mathcal{L}_{exp}$-sentences. Then in section \ref{Decidability of the LPEC-sentences}, we will see how one can adapt the arguments to a general existential $\mathcal{L}_{pEC}$-sentence.

\par Let $\psi$ be an existential $\mathcal{L}_{exp}$-sentence. It is sufficient to find a procedure that will return true when $\psi$ is true in $\Zp$ (and may runs forever when $\psi$ is false as by effective model-completeness the negation of $\psi$ is also an existential $\mathcal{L}_{pEC}$-sentence). An easy case is when $\psi$ is a formula of the type
$$\exists x_1,\cdots, x_n P_1(\overline{x},E_p(\overline{x})) = \cdots = P_n(\overline{x},E_p(\overline{x}))= 0 \not= J(\overline{x}) $$
where $J$ denote the Jacobian of the system.

In the real case, one can check that such a system of $n$ exponential equations with $n$ variables has a solution in $\R$ using Newton algorithm. In our case, we will use an analytic version of Hensel's lemma.

\begin{theorem}[Analytic Hensel's lemma]\label{Hensel}
Let $f= (f_1 , \cdots ,f_t)$ be a system of elements in $\Zp\{X_1,\cdots , X_t\}$. Let $J_f(\overline{X})$ denote the Jacobian matrix of the system. Assume that there is $\overline{a}\in \Zp^t$ such that 
$$\mbox{det }J_f(\overline{a})\not=0\mbox{ and }v(f(\overline{a}))> 2 v(\mbox{det }J_f(\overline{a})) + r,$$
 where $r$ is any nonnegative integer. Then, there is a unique $\overline{b}\in \Zp^n$ such that 
 $$v(\overline{b}-\overline{a})>v(\mbox{det }J_f(\overline{a}))+r\mbox{ and }f(\overline{b})=0.$$
\end{theorem}

 Assume that the formula $\psi$ is realised in $\Zp$, one can check it using Hensel's lemma:
\begin{enumerate}
	\item Let $\overline{k}_1,\overline{k}_2,\cdots$ be an enumeration of $\Z^n$.
	\item For all $i$, check if $\mbox{det }J_f(\overline{k}_i)\not=0\mbox{ and }v(f(\overline{k}_i))> 2 v(\mbox{det }J_f(\overline{k}_i))$.
	\item If this is the case for some $i$, stops the procedure: the formula $\psi$ is satisfied in $\Zp^n$ by Hensel's lemma.
\end{enumerate}
Let us remark that whenever the formula $\psi$ is true, we can always find such a tuple $\overline{k}_i$ as $\Z^n$ is dense in $\Zp^n$.

\par We will reduce the general case to a procedure of this kind. More precisely, let $F_P(\overline{x})=P(\overline{X},E_p(\overline{x}))$ be an exponential polynomial with coefficients in $\Z$. Assume that the formula $\exists \overline{x}\ F_P(\overline{x})=0$ is realised in $\Zp$. Then in section \ref{desingularization}, we will prove a desingularization theorem i.e. that there are  exponential polynomials $P_1(\overline{x},E_p(\overline{x})), \cdots , P_n(\overline{x},E_p(\overline{x}))$ and
 a non-singular solution of the system which also a zero of $F_P$. While checking that the system has a non-singular solution is easy, it is not obvious that the solution we find is indeed the zero of $F_P$ given by the theorem. Furthermore, we don't know any effective procedure to compute $P_1(\overline{x},E_p(\overline{x})) = \cdots = P_n(\overline{x},E_p(\overline{x}))$. This is where we need a $p$-adic version of Schanuel's conjecture. We will see in section \ref{Decidability of the LPEC-sentences} that if we assume the $p$-adic Schanuel's conjecture, we can actually assume that $P$ is (roughly) in the ideal generated by $P_1,\cdots P_n$. This will allow us to overcome the issues mentionned and give a positive answer to the decision problem for existential $\mathcal{L}_{pEC}$-sentences.
  
\begin{notation} Let $K$ be a valued field. We denote its valuation ring by $\mathcal{O}_K$ and its maximal ideal by $\mathfrak{M}_K$. We will denote the $p$-adic valuation by $v$. We will denote the set of nonzero elements of a ring $A$ by $A^*$.
\end{notation}

\section{Desingularization theorem}\label{desingularization}

Let $F$ be a subset of $\Zp\{\overline{X}\}$. Let $\mathcal{L}_F$ be the language $(+,\cdot , 0, 1, P_n, f; n\in \N, f\in F)$ where $P_n$ is a unary predicate interpreted in $\Zp$ by the set of $n$th powers. We assume that the set of $\mathcal{L}_F$-terms is closed under derivation. The example of our interest is the case where $F$ is the set of trigonometric functions (in the sense of \cite{Macintyre4}, we recall the definition later in this paper) and $E_p$.
\par In this section, we consider a system of equations $f=(f_1,\cdots , f_n)$ where the $f_i's$ are $\mathcal{L}_F$-terms with $m$ variables. Assuming that the above system has a solution in $\Zp$, we want to show that there exists a system of $\mathcal{L}_F$-terms $g=(g_1,\cdots ,g_m)$ such that there is a non-singular zero of the system $g$ which is also a zero of the system $f$.  We will actually prove the result for all finite algebraic extensions of $\Qp$. This result is the $p$-adic version of theorem 5.1 in \cite{Wilkie}. We will work  with Noetherian differential rings like in \cite{Wilkie}. The outline of the proof is actually the same that in the real case.
\par Within this section, $K$ will denote a finite algebraic extension of $\Qp$. The implicit function theorem will play an important role in our proof. We state now this result in the $p$-adic context.

\par Let $f:U\times V\longrightarrow K^m$ be an analytic map (where $U\times V$ is an open subset of $K^n\times K^m$). Let $A=Df(\overline{a},\overline{b})$ be its differential at $(\overline{a},\overline{b})$. We set
$$A_x=\left(\begin{array}{rcl}
\frac{\partial f_1}{\partial x_1}(\overline{a},\overline{b})& \cdots &\frac{\partial f_1}{\partial x_n}(\overline{a},\overline{b})\\
\vdots & &\vdots\\
\frac{\partial f_m}{\partial x_1}(\overline{a},\overline{b})& \cdots &\frac{\partial x_n}{\partial y_m}(\overline{a},\overline{b})
\end{array}\right)
 \mbox{ and } 
 A_y=\left(\begin{array}{rcl}
\frac{\partial f_1}{\partial y_1}(\overline{a},\overline{b})& \cdots &\frac{\partial f_1}{\partial y_m}(\overline{a},\overline{b})\\
\vdots & &\vdots\\
\frac{\partial f_m}{\partial y_1}(\overline{a},\overline{b})& \cdots &\frac{\partial f_m}{\partial y_m}(\overline{a},\overline{b})
\end{array}\right)
 $$
\begin{theorem}[Implicit function theorem] Let $f:U\times V\rightarrow K^m$ be an analytic map (where $U\times V$ is an open subset of $K^n\times K^m$) such that $f(\overline{a},\overline{b})=0$ for some $(\overline{a},\overline{b})\in U\times V$. Let $A=Df(\overline{a},\overline{b})$. Assume $A_y$ invertible. Then, there exist $U_1\subset U$ and $U_2\subset V$, both open and containing $\overline{a}$ and $\overline{b}$ respectively, such that for all $\overline{x}\in U_1$ there is a unique $\overline{y}\in U_2$ with $f(\overline{x},\overline{y})=0$.
\par Furthermore, the map $g$ defined by $g(\overline{x})=\overline{y}$ from $U_1$ to $U_2$ is analytic and satisfies $g(\overline{a})=\overline{b}$, $f(\overline{x},g(\overline{x}))=0$ for all $\overline{x}\in U_1$ and $Dg(\overline{x})=-A_y^{-1}A_x$.
\end{theorem}
\par Let us remark that if the function $f$ and the open sets $U, V$ are definable, then so is the function $g$. Indeed, we can assume that the $U_i$'s are open balls and the function $g$ is determined by the relations $(\overline{x},g(\overline{x}))\in U_1\times U_2$ and $f(\overline{x},g(\overline{x}))=0$. Also, the derivatives of $g$ are definable via the relation $Dg(\overline{x})=-A_y^{-1}A_x$.

\par We are now given a system $f_1,\cdots ,f_n$ of $\mathcal{L}_F$-terms. We first observe that such a system can be reduced to a single equation in $K$ : indeed, as $v(\mathcal{O}_K)=\frac{1}{e}\Z$ for some $e\in \N$:
$$ \mbox{for all $x,y\in K$, }(x,y)=(0,0) \mbox{ iff } x^2+\pi y^2=0,$$
where $\pi$ is an element of minimal positive valuation.
So, we can consider systems with a single $\mathcal{L}_F$-term. We view $K$ as a $\mathcal{L}_F$-structure (where $f\in F$ is interpreted by the map restricted to the valuation ring, i.e. the interpretation of $f$ takes value $f(x)$ for $x\in \mathcal{O}_K$ and value $0$ for $x\notin \mathcal{O}_K$). We also add to the language $\mathcal{L}_F$  constant symbols for a basis of $K$ over $\Qp$ (such that this basis is also a basis of $\mathcal{O}_K$ over $\Zp$). We are interested by the local behaviour of the definable analytic maps (especially, in what happens in the valuation ring). We consider the ring of such maps where we identify two maps which coincide on a open set i.e. the ring of germs:
\begin{definition}  
Given a \emph{neighbourhood system} $\mathcal{N}$ in $K^n$ (i.e. a non-empty collection of non-empty open $\mathcal{L}_F$-definable subsets of $K^n$ closed under finite intersection), $\mathfrak{G}^{(n)}(\mathcal{N})^-$ is the set of all $\langle f,U\rangle $ where $U\in \mathcal{N}$ and $f:U\longrightarrow K$ is a $\mathcal{L}_F$-definable function such that $f$ is analytic on $U$.
\par We define an equivalence relation on $\mathfrak{G}^{(n)}(\mathcal{N})^-$ by:
$\langle f_1,U_1\rangle\sim \langle f_2,U_2\rangle$ iff $f_1$ and $f_2$ coincide on a neighbourhood i.e. there is $U\in \mathcal{N}$ such that $U\subseteq U_1\cap U_2$ and for all $x\in U$, $f_1(x)=f_2(x)$. We denote by $[f,U]$ the class of $\langle f,U\rangle$.
\par The \emph{ring of germs} is the set $\mathfrak{G}^{(n)}(\mathcal{N})=\mathfrak{G}^{(n)}(\mathcal{N})^-/\!\!\sim$ equipped with the natural operations of addition and multiplication.
\end{definition}
Let us remark that $ \mathfrak{G}^{(n)}(\mathcal{N})$ is a unital differential ring.

\par As a special case of neighbourhood system, we have the collection of all definable open neighbourhoods of a point $P$. We denote the ring of germs in this case by $\mathfrak{G}^{(n)}(P)$. Let $P\in K^l$ and $Q\in K^m$ and let $f_1,\cdots , f_m$ be analytic maps in $\mathfrak{G}^{(l+m)}(P,Q)$. Let $f=(f_1,\cdots , f_m)$. Assume that $f(P,Q)=0$ and det $J_f(P,Q)\not =0$ i.e. $f_i(P,Q)=0$ for all $i$ and 
$$
\mbox{det}
\left(
\begin{array}{ccc}
\frac{\partial f_1}{\partial x_{l+1}} (P,Q) & \cdots & \frac{\partial f_1}{\partial x_{l+m}} (P,Q)\\
\vdots & & \vdots\\
\frac{\partial f_m}{\partial x_{l+1}} (P,Q) & \cdots & \frac{\partial f_m}{\partial x_{l+m}} (P,Q)
\end{array}\right)
 \not=0,
$$  
i.e. $f_i(P,Q)=0$ for all $i$ and  the vectors 
$$\left(\frac{\partial f_1}{\partial x_{l+1}} (P,Q),  \cdots , \frac{\partial f_1}{\partial x_{l+m}} (P,Q)\right) , \cdots , \left(\frac{\partial f_m}{\partial x_{l+1}} (P,Q) , \cdots , \frac{\partial f_m}{\partial x_{l+m}} (P,Q)\right)$$
 are $K$-linearly independent. We denote these vectors by $d_{P,Q}f_i$. By the analytic implicit function theorem, there are $U_P\times U_Q\subset U$ and $\Phi' =(\Phi_{l+1},\cdots ,\Phi_{l+m})$ analytic from $U_P$ to $U_Q$ such that $f_i(\overline{x},\Phi_{l+1}(\overline{x}),\cdots ,\Phi_{l+m}(\overline{x}))=0$ for all $\overline{x}\in U_P$.
 As the $f_i$'s and $U$ are definable, this guarantees that the map $\Phi'$ is definable analytic. Therefore, $\Phi'$ determines a germ in $\mathfrak{G}^{(l)}(P)$.
\par Let $\Phi(\overline{x}) =(\Phi_1(\overline{x}),\cdots , \Phi_{l+m}(\overline{x}))$ where $\Phi_i(\overline{x})=x_i$ for $i\leq l$ and $\Phi_{l+i}$ as above. We denote the morphism of rings
$$
\begin{array}{rll}
\mathfrak{G}^{(l+m)}(P,Q) & \longrightarrow & \mathfrak{G}^{(l)}(P) \\
 {[f,U]} & \longmapsto & [f(\Phi),U]
\end{array}
$$ 
by $\widehat{\phantom{123}}$. The kernel of this map is the set of germs which vanish (locally) on the set of zeros of the system $(f_1,\cdots ,f_m)$ around $(P,Q)$. In particular, $\hat{f}_i\equiv 0$ (and therefore, $\deriv{\hat{f}_i}{x_j}\equiv 0$) in $\mathfrak{G}^{(l)}(P)$.
\begin{lemma}\label{Wilkie lemma 4.7} Let $f_1,\cdots ,f_m$, $(P,Q)$ as above.
For all $g\in \mathfrak{G}^{(l+m)}(P,Q)$,\\ $d_{P,Q}f_1,\cdots d_{P,Q}f_m, d_{P,Q}g$ are linearly independent over $K$ iff $d_P\hat{g}\not=0$.
\end{lemma}

The proof of this lemma is word to word the same that lemma 4.7 in \cite{Wilkie}. We fix now some notations: let $f_1,\cdots , f_m: U\longrightarrow K$ be analytic functions (where $U\subset K^n$ open). Then,
$$\begin{array}{l}
V(f_1,\cdots ,f_m)=\{P\in U\mid\ f_1(P)=\cdots =f_m(P)=0\},\\
V^{ns}(f_1,\cdots , f_m)=\{P\in V(f_1,\cdots ,f_m)\mid\ d_Pf_1,\cdots ,d_Pf_m\mbox{ are $K$-linearly independent}\}.
\end{array}
$$

\begin{prop}\label{Wilkie thm 4.9} Let $P\in K^n$ and let $M$ be a Noetherian subring of $\mathfrak{G}^{(n)}(P)$ closed under differentiation. Let $m\in \N$ and $[f_1,U_1],\cdots ,[f_m,U_m]\in M$. Assume $P\in V^{ns}(f_1,\cdots , f_m)$. Then, exactly one of the following is true:
\begin{enumerate}[(a)]
\item n=m; or,
\item m<n and for all $[h,W]\in M$ with $h(P)=0$, $h$ vanishes on $U\cap V^{ns}(f_1,\cdots , f_m)$ for some $U$ open neighbourhood of $P$; or,
\item m<n and for some $[h,W]\in M$, $P\in V^{ns}(f_1,\cdots , f_m,h)$.
\end{enumerate}
\end{prop}
Again the proof is similar to the real case \cite{Wilkie}. Note that for this proposition, we need to consider analytic functions in our case (instead of infinitely differentiable functions in \cite{Wilkie}).
\begin{proof}
If $m<n$, say $n=l+m$, then the vectors $d_{P}f_1,\cdots , d_{P} f_m$ are linearly independent. Without loss of generality, we will assume that the matrix $A (P)=\left(\frac{\partial f_i}{\partial x_j}(P)\right)_{1\leq i\leq m, l+1\leq j\leq n}$ is invertible. Let $\lambda$ be the map $\overline{x}\longmapsto $det $A(\overline{x})$. On a neighbourhood $U$ of $P$, this map is invertible. Let $\Lambda = [\lambda ,U]$. We define $M^*:=M[\Lambda^{-1}]$. Assume $P=(P_1,P_2) \in K^{l\times m}$. We define the $\widehat{\phantom{123}}$-map as before. Then, $\widehat{M^*}$, the image of $M^*$ by this map, is Noetherian. And, by the implicit function theorem, we have 
$$\left(
\begin{array}{c}
\frac{\partial \Phi_{l+1}}{\partial x_r}\\
\vdots\\
\frac{\partial \Phi_{n}}{\partial x_r}
\end{array}
\right)
= -\Lambda^{-1}
\left(
\begin{array}{c}
\frac{\partial f_{1}}{\partial x_r}\\
\vdots\\
\frac{\partial f_{m}}{\partial x_r}
\end{array}
\right),$$
which means that $\deriv{\Phi_i}{x_j}\in M^*$. Therefore using the chain rule, we find that $\widehat{M^*}$ is closed under differentiation. 
\par  Let $I=\{ g\in \widehat{M^*}\mid\ g(P_1)=0\}$.
\begin{enumerate}
	\item If $I=\{0\}$. Suppose $g=[h,W]\in M$ and $h(P)=0$. Then, $\hat{g}(P_1)=0$ and therefore $\hat{g}\in I$ i.e. $\hat{g}=0$. By definition of the map $\widehat{\phantom{123}}$, it exactly means that $h$ is vanishing on a neighbourhood of $P$ in $V^{ns}(f_1,\cdots , f_m)$.
	\item If $I\not=\{0\}$, $I$ is not closed under differentiation. Otherwise for all $g\in I$, the partial derivatives of $g$ vanish at $P_1$. This implies that all the coefficients of the power series defining $g$ around $P$ are zero and therefore $g=0$ in $\widehat{M^*}$. So, there is $g\in M^*$ such that $\hat{g}\in I$ and $\frac{\partial \hat{g}}{\partial x_i}\notin I$. It means that $\hat{g}(P_1)=0$ (i.e. $g(P)=0$) and  $\frac{\partial \hat{g}}{\partial x_i}(P_1)\not=0$. But, for some integer $s$, $\Lambda^sg\in M$. Let $f=\Lambda^s g$. Then, $f(P)=0$ and 
	$$\frac{\partial \hat{f}}{\partial x_i}(P_1)=\left(s\hat{\Lambda}^{s-1}(P_1)\frac{\partial \hat{\Lambda}}{\partial x_i}\hat{g}(P_1)\right) + \left( \hat{\lambda}^s(P_1)\frac{\partial \hat{g}}{\partial x_i}(P_1)\right)\not=0.$$
	 So, $d_P\hat{f}\not=0$ and therefore by lemma \ref{Wilkie lemma 4.7}, $P\in V^{ns}(f_1,\cdots , f_m,f)$.
	
\end{enumerate}  
\end{proof}

We are now able to state the desingularization theorem:
\par Let $U$ be an open definable neighbourhood of the origin contained in $\mathcal{O}_K^n$. Then, $\{U\}$ forms a neighbourhood system. We denote the correspondent ring of germs by $\mathfrak{G}^{(n)}(U)$.  
Let us recall that we may assume $K=\Qp(\alpha_1,\cdots, \alpha_s)$ and that for this choice of $\alpha_1,\cdots , \alpha_s$,
 $\Z (\alpha_1 ,\cdots , \alpha_s)$ is dense in the valuation ring of $K$. 
\begin{theorem}\label{Wilkie thm 5.1}
Let $M$ be a Noetherian subring of $\mathfrak{G}^{(n)}(U)$ which contains $\Z (\overline{\alpha})[x_1,\cdots ,x_n]$, is closed under differentiation and such that for all $g\in M$, the germ of $g$ is equivalent to a definable analytic function given by a power series  with coefficients in the valuation ring.
\par Let $f\in M$. Assume that $S$ is a non-empty definable subset of $V(f)$, open in $V(f)$. Then, there exist $f_1,\cdots , f_n\in M$ such that $S\cap V^{ns}(f_1,\cdots , f_n)\not=\varnothing$.
\end{theorem} 

\begin{proof}
First, for all $Q\in S$, we set $I_Q=\{g\in M\mid\ g(Q)=0\}$. As $M$ is Noetherian, there is some $R$ in $S$ such that $I_R$ is maximal within the collection of all $I_Q$. Let $g_1,\cdots ,g_N$ be generators of $I_R$ and $g=\sum_i \pi^{2(i-1)}g_i^2$ (where $\pi$ is a prime element of $K$ which can be assumed to be one of the $\alpha_i$). So, $g(x)=0$ iff $g_i(x)=0$ for all $i$. Then, $R\in V(g)\cap S$ and for all $Q\in V(g)\cap S$, $I_R=I_Q$.
Choose $m$ maximal such that for some $f_1,\cdots ,f_m\in M$, $R\in V^{ns}(f_1,\cdots ,f_m)$.
\par By contradiction, assume that $m<n$; say $n=m+l$.

\par Note that up to a $\Z(\overline{\alpha})$-linear change of variables, we may assume $R$ as close to the origin as we will need (more precisely, we need that the neighbourhood $W_R$ below contains the origin). First, we will now prove that $V(g)\cap S$ and $V^{ns}(f_1,\cdots ,f_m)$ locally coincide.
\begin{enumerate}[(a)]
	\item $V(g)\cap S\subseteq V^{ns}(f_1,\cdots , f_m)$:
\par	Indeed, $R\in V^{ns}(f_1,\cdots , f_m)$. So, $f_i\in I_R$ for all $i$ and det $E\notin I_R$ (where $E$ denotes the matrix $\left( \frac{\partial f_i}{\partial x_j}\right)$ with $K$-linearly independent vectors). As, for all $Q\in V(g)\cap S$, $I_Q=I_R$, it means that  $f_i\in I_Q$ and det $E\notin I_Q$. So, $Q\in V^{ns}(f_1,\cdots , f_m)$.  
	\item Let $Q\in V(g)\cap S$, $h\in  M$ then $Q\notin V^{ns}(f_1,\cdots , f_m,h)$:	
\par	If we assume $Q\in V^{ns}(f_1,\cdots , f_m,h)$, arguing like in (a), we would find $R\in V^{ns}(f_1,\cdots , f_m,h)$ which contradicts the maximality of $m$.
	\item For all $Q\in V(g)\cap S$, there is $W_Q$ an open neighbourhood of $Q$ such that $W_Q\cap V(g)\cap S=W_Q\cap V^{ns}(f_1,\cdots , f_m)$:
\par 	By the point (b) and the proposition \ref{Wilkie thm 4.9}, the only possibility is that there is $W'$ open neighbourhood of $Q$ such that $g$ vanishes on $W'\cap V^{ns}(f_1,\cdots , f_m)$. As $f\in I_R$, it means that $V(f)\supseteq V(g)$ and therefore that $f$ vanishes on $W'\cap V^{ns}(f_1,\cdots , f_m)$. So, $W'\cap V^{ns}(f_1,\cdots , f_m)\subseteq W'\cap V(g)\cap V(f)$. We have that $S$ is open in $V(f)$. So, for some $W''$ open neighbourhood of $Q$, $W''\cap S=W''\cap V(f)$. Take $W_Q=W'\cap W''$ and we are done.
\end{enumerate}

We are given $f_1,\cdots , f_m$. Without loss of generality, we may assume that the matrix  $\Delta = \left( \frac{\partial f_i}{\partial x_j}\right)_{1\leq i\leq m; l+1\leq j\leq n}$ has non-vanishing determinant at $R=(P,Q)$. Let $\Phi_{l+1}(\overline{x}),\cdots , \Phi_n(\overline{x})$ given by the implicit function theorem and let $\Phi_i (\overline{x}) = x_i$ for $i\leq l$.
\par First, let us remark that up to a change of variables, we can assume that $\Phi_i(P)$ and $\frac{\partial \Phi_i}{\partial x_j}(P)$ (where $i>l\geq j$) lies in the  maximal ideal $\mathfrak{M}_K$.
 Indeed, by a change of variables of the type $(X_1,\cdots , X_n)\longmapsto (X_1-N_1,\cdots , X_n-N_n)$ (where $N_i\in  \N(\overline{\alpha})$ is a suitable approximation of $P_i,Q_i$), we can assume that  the implicit functions are defined on a neighbourhood of $0$. This means that we can assume $v(P)>t$ and $v(\Phi_i(P))=v(Q)>t$ (where $t$ could be any nonnegative integer). Also, we know that for all $r\leq l$

\begin{equation}
\left(
\begin{array}{c}
\frac{\partial \Phi_{l+1}}{\partial x_r}\\
\vdots\\
\frac{\partial \Phi_{n}}{\partial x_r}
\end{array}
\right)(\overline{X})
= -\Delta^{-1}
\left(
\begin{array}{c}
\frac{\partial f_{1}}{\partial x_r}\\
\vdots\\
\frac{\partial f_{m}}{\partial x_r}
\end{array}
\right)(\overline{X},\Phi_{l+1}(\overline{X}),\cdots , \Phi_{n}(\overline{X})). \label{derivative relation} 
\end{equation}

We consider the change of variables $$(X_1,\cdots , X_l,X_{l+1},\cdots, X_n) \longmapsto (X_1,\cdots , X_l,X_{l+1}/\pi^t,\cdots, X_n/\pi^t).$$ Denote by $\tilde{f}$ the function obtained after this change of variables. Then, for all $i\leq m$,
$$
\frac{\partial \tilde{f}_i}{\partial x_j} (\widetilde{P},\widetilde{Q})= 
\left\{\begin{array}{rl}
\frac{\partial f_i}{\partial x_j}(P,Q)&\mbox{ for $j\leq l$}\\
 \pi^{-t} \frac{\partial f_i}{\partial x_j}(P,Q)&\mbox{ for $l+1\leq j\leq n$},
 \end{array}\right.
$$
where $(\widetilde{P},\widetilde{Q})=(P,\pi^tQ)$. So, $\tilde{\Delta}(\widetilde{P},\widetilde{Q}) = \frac{1}{\pi^t} \Delta (P,Q)$. For $t$ large enough, $\tilde{\Delta}(\widetilde{P},\widetilde{Q})$ has negative valuation. Therefore, by the relation (\ref{derivative relation}), 
$$v\left(\frac{\partial \widetilde{\Phi}_i}{\partial x_j}(\widetilde{P},\widetilde{Q})\right)>0.$$
Without loss of generality, we will assume that such a changement has be done and will denote $(\widetilde{P},\widetilde{Q})$ by $(P,Q)$ and similarly for the functions.

\par Let $h_N(\overline{X}):=\sum (X_i-N_i)^2$ for all $N=(N_1,\cdots, N_n)\in \Z[\alpha_1 ,\cdots , \alpha_s]^n$. We want to apply Hensel's lemma to the functions $\left(\frac{\partial \hat{h}_N}{\partial X_1},\cdots ,\frac{\partial \hat{h}_N}{\partial X_l}\right)$. 
Our goal is to prove that for a point $(P',Q')$, close enough from $(P,Q)$, the vectors $d_{(P',Q')}f_1,\cdots , d_{(P',Q')}f_m, d_{(P',Q')}h_N$ are linearly dependent. For this, by lemma \ref{Wilkie lemma 4.7}, it is sufficient to check that the above partial derivatives vanish at $P'$.
\par We want to prove that if we choose $N$ carefully, for all $i$, $\frac{\partial \hat{h}_N}{\partial X_i}(P)$ has valuation at least $2v($ det $J(P))+\varepsilon+1$ where $J$ is the Jacobian of the system, $\varepsilon$ is the radius of the open set $W_R$ given in (c). Then, the analytic Hensel's lemma gives us a root $P'$, $\varepsilon$-close from $P$.
\begin{claim} det $J(P)\not=0$.
\end{claim}
\begin{proof}
 We compute the following derivatives using the chain rule :
$$ g_i:=\frac{\partial \hat{h}_N}{\partial X_i}(P) = \sum_{1\leq k \leq n} 2\cdot \Big(\Phi_k (P)-N_k\Big) \frac{\partial \Phi_k}{\partial X_i} (P).$$
$$ \deriv{g_i}{X_j}(P) = \frac{\partial^2 \hat{h}_N}{\partial X_j\partial X_i}(P)=  \sum_k 2\cdot \Big( \frac{\partial \Phi_k}{\partial X_i} (P)\cdot  \frac{\partial \Phi_k}{\partial X_j} (P)\Big) + 2\cdot \Big(\Phi_k (P)-N_k\Big) \frac{\partial^2 \Phi_k}{\partial X_j\partial X_i} (P).$$
     
We want to prove that the Jacobian of $g=(g_1,\cdots g_n)$ is non vanishing at $P$. In the above sum, let us denote  $\sum_k 2\cdot \frac{\partial \Phi_k}{\partial X_i} (P)\cdot \frac{\partial \Phi_k}{\partial X_j} (P)$  by $B_{ij}$ and the other terms $\sum_k 2\cdot \Big(\Phi_k (P)-N_k\Big) \frac{\partial^2 \Phi_k}{\partial X_j\partial X_i} (P)$ by $C_{ij}$. Then, let $S_l$ be the permutation group of $\{1,\cdots ,l \}$ and $sgn(\sigma )$ be the signature of an element $\sigma\in S_l$. We have:
\begin{align*}
\mbox{det } J_g(P) &= \sum_{\sigma\in S_l} sgn(\sigma)  \prod J_{i\sigma (i)}\\
								 &= \sum sgn(\sigma ) \prod (B_{i\sigma (i)}+ C_{i\sigma (i)})\\
								 &= \mbox{det }B + (\cdots),	 
\end{align*}
  
where in the sum $(\cdots)$, each element contains at least one factor of the form $\Big(\Phi_k (P)-N_k\Big)$.
\par If det $B\not=0$, then, for $N_k$ a suitable approximation of  $\Phi_k (P)$, the valuation of det $J_g(P)$ is given by the valuation of det $B$ (let us remark that in this case this valuation does not depend on $N$). And therefore, det $J_g(P)\not=0$.
\par We remark that for all $k\leq l$, $\frac{\partial \Phi_k}{\partial X_i}\cdot \frac{\partial \Phi_k}{\partial X_j} = \delta_{ijk}$ ($\delta$ is the Kronecker delta). So, if we denote by $D_{ij}$ the sum over $k>l$ in $B_{ij}/2$, we have: $\frac{1}{2}B= Id+ D$ and
\begin{align*}
\frac{1}{2}\mbox{det } B &= \sum_{\sigma\in S_r} sgn(\sigma) . \prod \Big(\delta_{i\sigma(i)} + D_{i\sigma(i)}\Big)\\
								 &= \left(\mbox{det }D - \prod D_{ii}\right) + \prod\Big(1+D_{ii}\Big).	 
\end{align*}   

Now, assume by contradiction that det $B =0$. Let us recall that for all $i>l$, for all $k$,
$$v\left(\frac{\partial \Phi_i}{\partial X_k}(P)\right)>0\ (*). $$
Therefore, $v(D_{ii})>0$ and as det $B=0$, 
 $$v\left( \mbox{det }D - \prod D_{ii}\right) = v \left(\prod(1+D_{ii})\right)=0.$$ 
We deduce from these relations that $v($ det $D)=0$. This is a contradiction with $(*)$.
\par This completes the proof of the claim.
\end{proof}

Now, for $N_k\in \Z[\alpha_1 ,\cdots , \alpha_s]$ a suitable approximation of $\Phi_k (P)$,  $g_i(P)$ has valuation at least $2v($ det $J(P))+\varepsilon+1$ (as we have seen the valuation of $J(P)$ does not depend on $N$ in this case). So, by Hensel's lemma, there exists $P'$ ($\varepsilon$-close from $P$) such that for all $i$, $g_i(P')=0$ i.e $d_{P'}\hat{h}_N =0$.
\par Let $Q'= (\Phi_{l+1}(P'),\cdots , \Phi_n(P'))$.  Then, $(P',Q')\in V^{ns}(f_1,\cdots , f_m)$ and  by lemma \ref{Wilkie lemma 4.7}, $d_{(P',Q')} f_1, \cdots, d_{(P',Q')} f_m, d_{(P',Q')} h_N$ are linearly dependent over $K$.
 But, as $(P',Q')$ is in $W_R$ (if we pick $\varepsilon$ small enough), we have that $(P',Q')\in V^{ns}(f_1,\cdots , f_m)\cap W_R\subseteq V(g)\cap S$. Then, by an argument similar to the proof of (a), $d_{(P,Q)} f_1, \cdots, d_{(P,Q)} f_m, d_{(P,Q)} h_N$ are also linearly dependent for all $N$ suitable approximation of $\Phi (P)$. As $d_{(P,Q)} f_1, \cdots, d_{(P,Q)} f_m$ are linearly independent, it implies that $d_{(P,Q)} h_N$ lies in the linear span of the other vectors.
\par Let $N'= (N_1,\cdots , N_{i-1},N_i+p^{t_i},N_{i+1},\cdots,N_n)$, then $N'$ is also a suitable approximation of $\Phi(P)$ (for all $t_i$ large enough) and therefore $d_{(P,Q)}h_{N'}$ lies in the same vector space. But then, $(0,\cdots , p^{t_i},0,\cdots , 0) = (d_{(P,Q)}h_{N'}-d_{(P,Q)}h_{N})/2$ lies in the linear span of  $d_{(P,Q)} f_1, \cdots, d_{(P,Q)} f_m$ for all $i$, which contradicts that $m<n$.   
\end{proof}
Our desingularization result is an immediate corollary of this theorem:
\begin{cor}\label{cor Wilkie 5.1}
Let $F\subset \mathcal{O}_K\{\overline{X}\}$. Assume that the set of $\mathcal{L}_F$-terms is closed under derivation and that for any finite collection of $\mathcal{L}_F$-terms, the ring generated by these terms, their subterms and their derivatives is Noetherian. Let $f$ be a $\mathcal{L}_F$-term and assume that the formula $\exists x_1\cdots  \exists x_n f(\overline{x})=0$ is satisfied in $\mathcal{O}_K$. Then there are $f_1,\cdots, f_n$ $\mathcal{L}_F$-terms such that the formula  $$\exists \overline{x}\ f(\overline{x})=f_1(\overline{x}) =\cdots = f_n(\overline{x})=0\not = det\ J_{(f_1,\cdots , f_n)}(\overline{x})$$ is satisfied in $\mathcal{O}_K$.
\end{cor}
\begin{proof}
 Let $f$ be a $\mathcal{L}_F$-term. Then, we apply the  theorem \ref{Wilkie thm 5.1} with $U=\mathcal{O}_K$, $M$ the ring generated over $\Z (\overline{\alpha})$ by the subterms of $f$ and their derivatives. We take $S=V(f)$. The theorem exactly says that if $V(f)\not=\varnothing$, then there are $f_1,\cdots , f_n\in M$ and $\overline{a}\in \mathcal{O}_K^n$ such that
$$f(\overline{a})=f_1(\overline{a}) =\cdots = f_n(\overline{a})=0\not = det\ J_{(f_1,\cdots , f_n)}(\overline{a}).$$
Finally, let us remark that as the set of $\mathcal{L}_F$-terms is closed under derivation, any element of $M$ is a $\mathcal{L}_F$-term.
\end{proof}
\begin{cor}\label{Desingularization exponential polynomial}
Let $P(X_1,\cdots, X_n,Y_1,\cdots , Y_n)\in \Z[\overline{X},\overline{Y}]$. Assume that $f(X)=P(\overline{X},E_p(\overline{X}))$ has a root in $\Zp$. Then there is $P_1,\cdots, P_n\in \Z[\overline{X},\overline{Y}]$ and $\overline{b}\in \Zp^n$ such that
$$P(\overline{b},E_p(\overline{b}))=P_1(\overline{b},E_p(\overline{b}))=\cdots = P(\overline{b},E_p(\overline{b}))=0\not= \mbox{ det }J_{(P_1,\cdots P_n)}(\overline{b}). $$
\end{cor}
\begin{proof} Apply theorem \ref{Wilkie thm 5.1} with $U=\Zp$, $M=\Z[\overline{X}, E_p(\overline{X})]$ and $S=V(P(\overline{X},E_p(\overline{X})))$.
\end{proof}

\section{Decidability of positive existential sentences}\label{Decidability of positive existential sentences}
\par It is easy to see that any existential $\Lexp$-sentence is (effectively) equivalent to a disjunction of sentences of the type:
$$ \exists x_1 \cdots \exists x_n  \bigwedge_j F_j(\overline{x})=0 \wedge \bigwedge_j G_j(\overline{x})\not=0,$$
where $F_i$ and $G_j$ are in $\Z[x_1,\cdots , x_n, e^{px_1},\cdots , e^{px_n}]$. Note that unlike the real case, we cannot remove the inequalities (in the real case, we can remove inequality at the price of new varibles because nonzero elements are invertible). 
\par Let us remark that to any such exponential polynomial $F$ corresponds a polynomial in $\Z[x_1,\cdots , x_{2n}]$. And conversely, to a polynomial $P\in \Z[x_1,\cdots , x_{2n}]$ corresponds a unique element of $\Z[x_1,\cdots , x_n, e^{px_1},\cdots , e^{px_n}]$. We will denote by $F_P$ this exponential polynomial.
\par We start with the case where only equalities are involved. Once again, as for all $x,y\in\Zp$ $(x,y)=(0,0)$ iff $x^2+py^2=0$, we are reduced to the case of a single exponential polynomial, say $F_P(x_1,\cdots ,x_n)$. First, note that it is not hard to check that a given positive existential sentence is false in $\Zp$: 
\begin{lemma}\label{universal formula stop}
 Let $U=\overline{a}+p^t\Zp^n$ be an open set (where $\overline{a}\in \Z^n$, $t\in \N$) and let $g=(g_1,\cdots ,g_k)$ be $k$ exponential functions. Then, there is a recursive procedure which returns yes if there is no zero of $g$ inside $U$.
\end{lemma} 
\begin{proof}
We just have to check that the valuation of $g$ is bounded on $U$.
\par Let us remark that if there is $\overline{y}\in U$ such that $g(\overline{y})=0$, then for all $s\geq t$, there are $b_{0i},\cdots , b_{si}\in \{0,\cdots , p-1\}$, $i\leq n$ such that 
 $$a_i\equiv \sum_j b_{ji}p^j\mod p^{t} \mbox{ and } g\left(\sum b_{j1}p^{j},\cdots , \sum b_{jk}p^{j}\right) \equiv 0 \mod p^{s+1}.$$ 
 Actually, the $b_{ji}$'s are the digits of $b_i$ a suitable approximation of $y_i$. 
\par So, the converse states that: if there is $s\geq t$ such that for all  $b_{0i},\cdots , b_{si}\in \{0,\cdots , p-1\}$ such that for all $i\leq n$
$$a_i\equiv \sum_j b_{ji}p^j\mod p^{t},$$
  we have
  $$g\left(\sum_j b_{j1}p^{j},\cdots , \sum_j b_{jk}p^{j}\right) \not\equiv 0 \mod p^{s+1},$$
 then, there is no $\overline{y}\in U$ such that $g(\overline{y})=0$.
\par  But these last conditions are recursively enumerable. The following algorithm does the job:
\begin{algo}\label{no root}
Given $U=\overline{a}+p^t\Zp^n, g=(g_1,\cdots ,g_k)$.
\par Proceed to an enumeration of all $s\in \N$, $s\geq t$. Check if for all $b_1,\cdots ,b_n\in \Z/p^s\Z$ with $v(\overline{a}-\overline{b})\geq t$, $g_i(\overline{b})\not\equiv 0\mod p^{s+1}$ for all $i$. If yes, return true. Otherwise go to the next step.
 \end{algo}
 This completes the proof of the lemma.
\end{proof}
Let us remark that this algorithm never stops if the system has a root in $U$. Note that we are able to do the basic computations in the above lemma (and in the rest of the paper) effectively. Indeed, we are able to compute the valuation of an exponential polynomial evaluated at a given integer.
 \par Let $f\in \Z[x_1,\cdots x_n,e^{px_1},\cdots ,e^{px_n}]$ and a tuple of integer $\overline{t}$. Then, we are able to determine if $f(\overline{t})=0$ and compute the valuation of $f(\overline{t})$:
Let us remark that we can assume that $f(\overline{t})$ is a finite sum of the form
$$ f(\overline{t}) = s \sum a_i e^{pi}$$
where the $a_i$'s are integer and $s\in \Zp^*$. As, $e^{p}$ is transcendental over $\Q$ (theorem due to Malher \cite{Malher}), $f(\overline{t})=0$ iff $a_i=0$ for all $i$. If this is not the case, using the Taylor expansion, we can determine the remainder of $f(\overline{t})$ modulo $p^n$ for all $n$. The valuation of $f(\overline{t})$ is determined by the smallest $n$ such that $f(\overline{t})\not\equiv 0 \mod p^n$.

\par Let $\psi\equiv \exists \overline{x}\ F_P(\overline{x})=0$ be a positive exisential formula. The above lemma gives us a procedure which stops and returns true if $\psi$ is false in $\Zp$. It runs forever if $\psi$ is true in $\Zp$. To complete the proof of the decidability of the set of positive existential formula, we just have to give a procedure that returns true whenever $F_P$ has a root in $\Zp$ (and may runs forever otherwise). Actually, using the desingularization theorem, we can almost already determine if $F_P$ has a root in $\Zp$:
\par Assume that $F_P$ admits a root $\overline{a}$ in $\Zp^n$, then we know by the corollary \ref{Desingularization exponential polynomial} that there are $F_{P_1},\cdots, F_{P_n}$ and $\overline{b}$ such that $F_P(\overline{b})=0$ and $\overline{b}$ is a non-singular zero of the system $G = (  F_{P_1},\cdots, F_{P_n})$. Now, by Hensel's lemma, we can check if the system $G$ has a non-singular solution in $\Zp$. However, there is two issues: first, the system $G$ is not given explicitely in the proof of the desingularization theorem (i.e. may not be computable). Second, there is no guarantee that the non-singular solution of the system $G$ computed by Hensel's lemma is indeed a solution of $F_P$.
\par The two issues can be solved if $P$ is in the ideal generated by $P_1,\cdots ,P_n$. Indeed, in that case, $\overline{b}$ determines a zero of each $P_i$:
 $$P_i(b_1,\cdots, b_n, e^{pb_1},\cdots , e^{pb_n})=0\mbox{ iff } F_{P_i}(b_1,\cdots , b_n)=0.$$
And conversely, any zero $b_1,\cdots, b_n, e^{pb_1},\cdots , e^{pb_n}$ of the system $P_1,\cdots, P_n$ is a zero of $P$. Furthermore, we can find $P_1,\cdots, P_n$ by enumerating all polynomials in $\Z[\overline{X}]$ (see later for more details).
\par The next lemma will give us exactly what we need : up to multiplication by a polynomial $Q$ (such that $Q$ does not vanish at $(b_1,\cdots, b_n, e^{pb_1},\cdots , e^{pb_n})$), $P$ is in the ideal generated by some $Q_1,\cdots , Q_n$ like above.
 The key point of this lemma is that we can determine the transcendence degree of  $\Q (b_1,\cdots, b_n, e^{pb_1},\cdots , e^{pb_n})$ over $\Q$. This is where Schanuel's conjecture turns out to be helpful.
\par  The first thing to observe is that as $\overline{b}$ is a non-singular zero of the system $G$, we certainly have that
 $$ \mbox{trdeg}_\Q \Q(b_1,\cdots, b_n, e^{pb_1},\cdots , e^{pb_n}) \leq n.$$
 We will actually need equality which can be obtained using a $p$-adic version of Schanuel's conjecture:
 
 \begin{conj}[$p$-adic Schanuel's Conjecture]
 Let $n\geq 1$ and $t_1,\cdots , t_n$ in $\Cp$ (with valuation at least $1/(p-1)$) linearly independent over $\Q$.
 \par Then, the field $\Q (t_1,\cdots ,t_n,e^{t_1},\cdots , e^{t_n})$ has transcendence degree at least $n$ over $\Q$.
 \end{conj}
  
Using the $p$-adic version of Schanuel's conjecture, like in \cite{Macintyre-Wilkie}, we can prove:
\begin{lemma}\label{lemma algo}
Let $n\geq 1$, $P\in\Z[x_1,\cdots , x_{2n}]$.
\par Assume that $F_P = P(x_1,\cdots , x_n,e^{px_1},\cdots , e^{px_n})$ has a zero and that for all zeros $\overline{a}$ of $F_P$, its components $a_1,\cdots , a_n$ are $\Q$-linearly independent.
\par Then, there exist $b_1,\cdots, b_n\in \Zp$ and $Q,Q_1,\cdots ,Q_n,S_1,\cdots S_n\in \Z[x_1,\cdots , x_{2n}]$ such that $\overline{b}$ is a zero of $F_P$ and a non-singular zero of $G=(F_{Q_1},\cdots, F_{Q_n})$, that $F_Q(\overline{b})\not=0$ and that $QP=\sum_i Q_iS_i$.
\end{lemma}
This lemma guarantees the existence of a system such that the non-singular zeros of this system are roots of $F_P$.
\begin{proof}
Let $P_1,\cdots, P_n$ given by corollary \ref{Desingularization exponential polynomial} and $\overline{b}\in V(F_P)\cap V^{ns}(F_{P_1},\cdots , F_{P_n})$.
\par By the above discussion, the transcendence degree of $\Q (b_1,\cdots, b_n, e^{pb_1},\cdots , e^{pb_n}) $ over $\Q$ is exactly $n$. We apply the following claim with $m=2n$, $r=n$ and $I= \{h\in\Z[x_1,\cdots, x_{2n}]\mid$ $h(b_1,\cdots, b_n, ^{pb_1},\cdots , e^{pb_n})=0\}$:
\begin{claim}\label{ideal generators}Let $m,r\geq 1$, $I$ prime ideal of $\Z[x_1,\cdots , x_m]$ with $I\cap\Z=\{ 0\}$ and trdeg$_\Q$ Frac $(\Z[x_1,\cdots ,x_m]/I)=r$.
Then, there is $Q\in \Z[\overline{x}]$ with $Q\notin I$ such that $QI$ is generated by $m-r$ elements.
\end{claim} 
As trdeg Frac$(\Z[\overline{x}]/I)=$ trdeg$_\Q \Q(b_1,\cdots , b_n E_p(b_1),\cdots , E_p(b_n))=n$ (by Schanuel's conjecture), we can apply the claim. 
\par Let $Q_1,\cdots , Q_n$ be generators of $QI$. Then, the properties of the lemma are satisfied except that $\overline{b}$ may be a singular zero of our system. But, as $P_i\in I$, $QP_i = \sum S_{ij}Q_j$ for some $S_{ij}\in \Z[\overline{x}]$. Using the chain rule on this relation, we find that
$$ F_Q(\overline{b}) \cdot \frac{\partial F_{P_i}}{\partial x_j}(\overline{b}) = \sum_k F_{S_{ik}}(\overline{b}) \frac{\partial F_{Q_k}}{\partial x_j}
 (\overline{b}). $$
As $F_Q(\overline{b}) \not=0$, we deduce that $\overline{b}$ is a non-singular zero of $G$. 
\end{proof}

\begin{prop}\label{positive existential theory}
If Schanuel's conjecture is true, the positive existential theory of the structure $(\Zp, +,\cdot ,0,1, E_p)$ is decidable.
\end{prop}
\begin{proof}
Let $\varphi$ be a positive existential sentence of our theory. Without loss of generality, we can assume that $$\varphi\equiv \exists x_1\cdots x_n\ F_P(\overline{x})=0$$
for some $P\in \Z[X_1,\cdots , X_{2n}]$.
\par First, we give an algorithm that returns true if the sentence is satisfied (and never stop otherwise). We are given $F_P$ and we want to know if this function admits a solution in $\Zp^n$. Assume that this is the case. Then, lemma \ref{lemma algo} gives us  the existence of exponential polynomial functions $G= (F_{Q_1}, \cdots, F_{Q_n})$ such that any non-singular zero of $G$ is a zero of $F_P$. So, proceed to an enumeration over all possible system $G$ and polynomials $Q$, $S_{ij}$ like in the lemma. Using Hensel's lemma, we can determine if $G$ has a non-singular root in an open $U$.  If our sentence is satisfied, there exists such an open set $U$ which contains a solution of $G$ and  does not contain a root of $F_Q$. So, we proceed to an enumeration of all open set of the type $U=\overline{a}+p^t\Zp^n$ for all $\overline{a}\in \N^n, t\in \N$ and on each such a set we check if the conditions of Hensel's lemma are satisfied for some tuple in $U$ and if $F_Q$ has no root in $U$ (via lemma \ref{universal formula stop}).
\par We give now the algorithm. If Schanuel's conjecture is true, this algorithm returns true whenever $F_P$ has at least one root in $\Zp^n$ and the components of any of its roots are linearly independent. If these conditions are not satisfied, this algorithm may run forever.
\begin{algo}\label{existence of root}
Given $n\geq 1, P\in\Z[x_1,\cdots , x_{2n}]$.\\
Proceed to an enumeration of $Q,Q_1,\cdots ,Q_n,S_1,\cdots, S_n\in$ $\Z[x_1, \cdots ,x_{2n} ]$  and all $a_1,\cdots , a_n,t,s \in \N, s\geq t$\\
Given such a $3n+3$-uple, first check if $QP=\sum Q_iS_i$. If not go to the next step (of the enumeration).\\
Otherwise, check if
$$\mbox{det }\left( \deriv{F_{Q_i}}{x_j}\right)(\overline{a})\not=0,$$
 and if $$v(F_{Q_i})(\overline{a})> v\Big(\mbox{det }\left( \deriv{F_{Q_i}}{x_j}\right)(\overline{a})\Big)+t.$$ 
 If not, go to the next step.\\
If this is the case (there is a root of the system $G$ in $U:=\overline{a}+p^t\Zp^n$), for all $b_{ij}\in \{0,\cdots, p-1\}$, where $0\leq  j \leq s$, $1\leq i\leq n$, let $b_i=\sum b_{ij}p^j$  check if whenever $v(\overline{b}-\overline{a})\geq t$, we have 
$$F_{Q}(\overline{b})\not\equiv 0 \mod p^{s+1}.$$
If yes ($F_Q$ does not admit root in $U$), return true. Otherwise, go to the next step.
\end{algo}
Finally, let us recall that in the above algorithm, we need to assume that the components of any root of $F_P$ are linearly independent. But, without loss of generality, we can  assume that this is the case:
\par  Indeed, let $F_P$ be an exponential polynomial. We proceed to an enumeration over all possible relations of $\Z$-linear dependence between the variables and we run in parallel the following procedure:
\par For each relation, we remove one of the variable according to this relation. Let $\widetilde{F_P}$ be the exponential polynomial obtained after this transformation. We remark that $\widetilde{F_P}$ has a root iff $F_P$ has a root that satisfies the $\Z$-linear relation used to construct $\widetilde{F_P}$. We apply the algorithm \ref{existence of root} with entry $\widetilde{F_P}$. If the components of any root of $\widetilde{F_P}$ are linearly independent, then algorithm \ref{existence of root} returns true (in the case where $\widetilde{F_P}$ has a root) and the truth of our formula is determined. If $\widetilde{F_P}$ has a root with components linearly dependent, we restart the procedure with $F_P:=\widetilde{F_P}$.
\par This procedure stops and returns true in the case where $F_P$ has a root in $\Zp$.

\par Now, we can determine the truth of a positive existential sentence : we run in parallel the algorithm \ref{no root} and algorithm \ref{existence of root} with entries $P$. If $F_P$ has no root in $\Zp^n$, the algorithm \ref{no root} stops and we return false. If not, then $F_P$ has a root and algorithm \ref{existence of root} stops, in which case, we return true.
\end{proof} 
\begin{Remark}
It is not hard to see that the algorithms \ref{no root} and \ref{existence of root} can be adapted to determine the truth of positive existential sentences in $(\mathcal{O}_{K},+,\cdot , 0, 1, E_p)$ where $K$ is a finite algebraic extension of $\Qp$. 
\end{Remark}

Let us also remark that the algorithm \ref{existence of root} can be easily modified to take as entries general existential sentences. Indeed, such a sentence has the form 
$$ \exists x_1 \cdots \exists x_n  F_P(\overline{x})=0 \wedge \bigwedge_j F_{R_j}(\overline{x})\not=0.$$
Therefore, we just have to check that $F_{R_j}$ has no root in $U$ (exactly like we did for $F_Q$). However, it is not clear that we can find a procedure that stops if such a sentence is false. One could proof that the negation of the above formula is equivalent to a (computable) $\mathcal{L}_{exp}$-existential sentence. It is not known if such a model-completeness result exists. In order to avoid this issue, we extend proposition \ref{positive existential theory} to the expansion of language $\mathcal{L}_{pEC}$ of \cite{Macintyre4}. In this language, the theory of $\Zp$ is effectively model-completene. The decidability of the full theory is then clear: apply the above procedure in parallel for the sentence and (the $\mathcal{L}_{pEC}$ -existential sentence equivalent to) its negation. One of the two procedure has to stop and therefore determines the truth of the sentence in $\Zp$.

\section{Decidability of the $\mathcal{L}_{pEC}$-sentences}\label{Decidability of the LPEC-sentences}

\par First, let us recall the definition of the language used in \cite{Macintyre4} to obtain the model-completeness result. Let $K=\Qp(\alpha)$ be a finite algebraic extension of degree $d$ (where we assume $\alpha$ has nonnegative valuation and is algebraic over $\Q$). We want to expand $\mathcal{L}_{exp}$ by function symbols so that the function $E_p:\Zp(\alpha)\longrightarrow \Zp(\alpha)$ is definable in our language. As usual, one can define the ring $\Zp (\alpha)$ in $\Zp$ by identifying the former ring with its structure of $\Zp$-module. So, it is sufficient to be able to decompose $E_p(y)$ in the basis of $K$ over $\Qp$ for all $y\in \Zp(\alpha)$. By the multiplicative property of $E_p$ we actually just need to decompose $E_p(x\alpha^i)$ for all $i<d$ and for all $x\in \Zp$. We set 
$$E_p(x\alpha^i)= c_{0,i}(x)+c_{1,i}(x)\alpha+\cdots + c_{d-1,i}(x)\alpha^{d-1}.$$
The elements $c_{i,j}(x)$ define functions from $\Zp\rightarrow \Zp$ such that $c_{i,j}(X)\in \Zp\{X\}$. We call these functions the \emph{decomposition functions of $E_p$ in $K$}. Note that these functions does not depend on the choice of the basis. Indeed, let $\sigma\in Gal(K/\Qp)$ (the Galois group of $K$ over $\Qp$) then
$$E_p(x\sigma(\alpha)^i)= c_{0,i}(x)+c_{1,i}(x)\sigma(\alpha)+\cdots + c_{d-1,i}(x)\sigma(\alpha)^{d-1}.$$
In fact, just like the function $\sin$ and $\cos$ in $\C$, the functions $c_i$ can be expressed as a polynomial combination of $E_p(x\sigma(\alpha))$. Indeed, from the above equalities, we find that:
$$
\left(\begin{array}{c}
c_{0,i}(x)\\
\vdots\\
c_{d-1,i}(x)
\end{array}\right)
=
V^{-1}
\left(\begin{array}{c}E_p(x\sigma_0(\alpha))\\
\vdots\\
E_p(x\sigma_{d-1}(\alpha))
\end{array}\right) (*),
$$
where $V$ is the Vandermonde matrix of the roots of the minimal polynomial of $\alpha$ and $\sigma_i$ are the elements of $Gal(K_d/\Qp)$.

\par We fix a tower of finite algebraic extensions $K_1\subset K_2\subset \cdots$ so that
\begin{itemize}
\item $K_n$ is the splitting field of $Q_n(X)$ polynomial of degree $N_n$ with coefficients in $\Q$;
\item $K_n= \Qp(\beta_n)$ for all $\beta_n$ root of $Q_n$;
\item any extension of degree $n$ is contained in $K_n$ and its valuation ring is contained in $\Zp[\beta_n]$.
\end{itemize}
Note that by Krasner lemma, such a family of extension exists. We denote by $c_{ijn}$ the decomposition functions of $E_p$ in $K_n$.  Let $\mathcal{L}_{pEC}=\mathcal{L}_{exp}\cup\{c_{ijn}:\ i<N_j, j,n\in \N \}$. We will say the functions $c_{ijn}$ are the \emph{trigonometric functions}.

Let $\Z_{pEC}$ be the structure with underlying set $\Zp$ and natural interpretations for the language $\mathcal{L}_{pEC}$. In \cite{Mariaule1}, it is proved that 
\begin{theorem}\label{effective model completeness} $Th(\Z_{pEC})$ is effectively strongly model-complete.
\end{theorem}
Therefore, if we are able to solve the decision problem for $\mathcal{L}_{pEC}$-existential sentences then the theory of $\Zpex$ is decidable.

\par Let $\varphi$ be an existential $\mathcal{L}_{pEC}$-sentence with $n$ quantifiers. Then, there is $N$ such that any term of the formula has the form $$f(\overline{x})=P(\overline{x},e^{p\overline{x}},c_{0,1,N}(\overline{x}),\cdots , c_{d_N-1,d_N-1,N}(\overline{x}))=:F_P(\overline{x})$$
 for some $P(\overline{x}, \overline{y}_0,\cdots ,\overline{y}_{L_N})\in \Z[\overline{x}, \overline{y}_0,\cdots ,\overline{y}_{L_N}]$ (where $L_N=d_N^2-d_N$). Let us remark that the ring generated by the exponential and trigonometric functions is closed under derivation. Therefore, we can apply theorem \ref{Wilkie thm 5.1}:
\par  Let $f$ be a $\mathcal{L}_{pEC}$-term. Assume that $V(f)\not=\varnothing$. Then, there exist $Q_1,\cdots , Q_n\in\Z[\overline{x}, \overline{y}_0,\cdots ,\overline{y}_{L_N}]$ such that $V^{ns}(F_{Q_1},\cdots F_{Q_{n}})\cap V(f)\not= \varnothing$. This implies that there is a root $\overline{a}$ of $f$ such that:
\begin{align*} 
 \mbox{trdeg}_\Q \Q(a_1,\cdots, a_n, e^{pa_1},\cdots , e^{pa_n}, c_{0,1,d_N}(a_1),\cdots , c_{d_N-1,d_N-1,N}(a_n)) \\
 =  \mbox{trdeg}_\Q \Q(a_1,\cdots, a_n, e^{pa_1}, \cdots, e^{pa_n},e^{pa_1\beta_N} \cdots , e^{pa_n\beta_N^{d_N-1}})\leq d_N\cdot n,
 \end{align*}
 where the first equality holds, as by (*), the trigonometric functions can be expressed as a polynomial combinations of $E_p(x\beta_N^k)$.
Let us remark that $1, \beta_N, \cdots ,\beta_N^{d_N-1}$ are $\Qp$-linearly independent. Using the $p$-adic Schanuel's conjecture, we find that the above relation is actually an equality (if $a_1,\cdots , a_n$ are $\Q$-linearly independent). With this, we prove as before:
\begin{lemma}\label{L_pec crucial lemma}
Let $n\geq 1$, $P\in\Z[\overline{X}, \overline{Y}_0,\cdots , \overline{Y}_{L_N}]$.	
\par Assume that $F_P = P(x_1,\cdots , x_n,e^{px_1},\cdots , e^{px_n},c_{0,1,N}(x_1),\cdots ,c_{N,d_N-1,d_N-1}(x_n))$ has a zero in $\Zp$ and that the components of any zero of $F_P$ are $\Q$-linearly independent.
\par Then, there exist $a_1,\cdots, a_{n}\in \Zp$ and $R,R_1,\cdots ,R_{T_N},S_1,\cdots S_{T_N}\in \Z[\overline{X}, \overline{Y}_0,\cdots , \overline{Y}_{L_N}]$ (where $T_N=(L_N-d_N+2)\cdot n$) such that $\overline{a}$ is a zero of $F_P$ and a non-singular zero of a subsystem of $G=(F_{R_1},\cdots, F_{R_{T_N}})$, that $F_R(\overline{a})\not=0$ and that $RP=\sum_i R_iS_i$.
\end{lemma}
\begin{proof}
We apply claim \ref{ideal generators} with 
$$I= \{h\in\Z[\overline{X}, \overline{Y}_0,\cdots , \overline{Y}_{L_N}]\mid\ h(\overline{a},E_p(\overline{a}),c_{0,1,N}(\overline{a}),\cdots , c_{d_N-1,d_N-1,N}(\overline{a}))=0\},$$ $m=(L_N^2+2)\cdot n$ and $r=d_N\cdot n$. Then, by the claim, there exist $R\not\in I$,  $R_1,\cdots , R_{T_N}$ generators of $RI$ (where $T_N=m-r$) and $S_1,\cdots , S_{T_N}$ such that $RP=\sum R_iS_i$. Also, as $Q_i\in I$ for all $i$, like in lemma \ref{lemma algo}, it implies that $\overline{a}\in V^{ns}(F_{\widetilde{R_1}},\cdots , F_{\widetilde{R_n}})$ for some $\widetilde{R_1}\cdots ,\widetilde{R_n}\in \{R_1,\cdots , R_{T_N}\}$.
\end{proof}

If we are given $\varphi$ an existential $\mathcal{L}_{pEC}$-sentence of the form:
$$\exists x_1,\cdots , x_n F_P(\overline{x})=0\wedge \bigwedge_i F_{A_i}(\overline{x})\not=0,$$
it is quite easy to adapt the algorithm \ref{existence of root} to construct an algorithm that returns yes if the sentence is true in $\Zp$ (and never stops otherwise):
\begin{enumerate}
	\item Enumerate all $R,R_1,\cdots ,R_{T_N},S_1,\cdots S_{T_N}$ and $B= \overline{a}+p^k\Zp^n$.
	\item If $RP= \sum R_iS_i$, check if a subsystem $\widetilde{R_1}\cdots ,\widetilde{R_n}$ has a unique non-singular root in $B$ using Hensel's lemma.
	\item If this is the case, use the algorithm \ref{existence of root} to determine if the following formula is true in $V_N$:
	$$\exists \overline{x}\ \overline{x} \in B \wedge \overline{x} \in V^{ns}(\widetilde{R_1}\cdots ,\widetilde{R_n})\wedge \overline{x} \in V(R_1,\cdots ,R_{T_N}). $$
	We use the version of algorithm \ref{existence of root} for formulas in $K_N$ and in the above formula, we replace the trigonometric functions by their polynomial expression in exponential terms. Note that this procedure never stops if the above formula is false but it doesn't matter.
	\item If the above formula is true, then the system $R_1,\cdots ,R_{T_N}$ has a root in $\Zp^n\cap B$. So, $F_P$ has a root in $\Zp^n\cap B$. It remains to check that $F_{A_i}$ and $F_R$ have no root in $B$. If this is the case, $\varphi$ is true.
\end{enumerate}
Now, we use theorem \ref{effective model completeness} to obtain a sentence $\psi$ equivalent to the negation of the sentence $\varphi$. Surely, our algorithm stops either for $\varphi$ or $\psi$. We can therefore determine the truth value of $\varphi$ in $\Zp$ by running in parallel the algorithm for $\varphi$ and $\psi$.
\par The main theorem follows:
\begin{theorem}\label{main theorem}
 Assume that the $p$-adic version of Schanuel's conjecture is true. Then, the theory of $\Z_{pEC}$ is decidable.
\end{theorem}

\begin{Remark}
\begin{itemize}
	\item By the remark after proposition \ref{positive existential theory}, it is not hard to extend the above theorem to finite algebraic extensions of $\Qp$.
	\item In the above theorem, we actually use the $p$-adic Schanuel's conjecture for point in the algebraic closure of $\Qp$ (i.e. we don't need the case where the $\alpha_i$'s are proper points of $\Cp$).
	\item For sentences in one variable, one can solve the decision problem unconditionally (see \cite{MyThesis}).
\end{itemize}
\end{Remark}

\bibliographystyle{plain}
\bibliography{Biblio_p-adic}

\end{document}